\documentclass[a4paper,english,fontsize=11pt,parskip=half,abstract=true]{scrartcl}
\usepackage{babel}
\usepackage[utf8]{inputenc}
\usepackage[T1]{fontenc}
\usepackage[a4paper,left=20mm,right=20mm,top=30mm,bottom=30mm]{geometry}
\usepackage{amsmath}
\usepackage{amsthm}
\usepackage{amssymb}
\usepackage{enumerate}
\usepackage{thmtools}
\usepackage[bookmarks=true,
            pdftitle={},
            pdfauthor={Benjamin Sambale},
            pdfkeywords={},
            pdfstartview={FitH}]{hyperref}

\newtheorem{Thm}{Theorem} 
\newtheorem{Lem}[Thm]{Lemma}
\newtheorem{Prop}[Thm]{Proposition}

\newtheorem{Con}[Thm]{Conjecture}
\theoremstyle{definition}
\newtheorem{Def}[Thm]{Definition}
\numberwithin{equation}{section}

\allowdisplaybreaks[1]

\renewcommand{\phi}{\varphi}
\newcommand{\C}{\mathrm{C}}
\newcommand{\N}{\mathrm{N}}
\newcommand{\Z}{\mathrm{Z}}
\newcommand{\ZZ}{\mathbb{Z}}

\newcommand{\FF}{\mathbb{F}}

\newcommand{\Aut}{\operatorname{Aut}}

\newcommand{\Out}{\operatorname{Out}}
\newcommand{\pcore}{\mathrm{O}}
\newcommand{\GL}{\operatorname{GL}}
\newcommand{\AGL}{\operatorname{AGL}}
\newcommand{\SL}{\operatorname{SL}}

\newcommand{\PSL}{\operatorname{PSL}}

\newcommand{\F}{\mathrm{F}}

\newcommand{\Syl}{\operatorname{Syl}}

\newcommand{\id}{\operatorname{id}}

\newcommand{\diag}{\operatorname{diag}}

\mathchardef\ordinarycolon\mathcode`\:  
 \mathcode`\:=\string"8000
 \begingroup \catcode`\:=\active
   \gdef:{\mathrel{\mathop\ordinarycolon}}
 \endgroup

\title{Generalized bases of finite groups}
\author{Benjamin Sambale\footnote{Institut für Algebra, Zahlentheorie und Diskrete Mathematik, Leibniz Universität Hannover, Welfengarten 1, 30167 Hannover, Germany,
\href{mailto:sambale@math.uni-hannover.de}{sambale@math.uni-hannover.de}}}
\date{\today}

\begin{document}
\frenchspacing
\maketitle
\begin{abstract}\noindent
Motivated by recent results on the minimal base of a permutation group, we introduce a new local invariant attached to arbitrary finite groups. More precisely, a subset $\Delta$ of a finite group $G$ is called a \emph{$p$-base} (where $p$ is a prime) if $\langle\Delta\rangle$ is a $p$-group and $\C_G(\Delta)$ is $p$-nilpotent. Building on results of Halasi--Maróti, we prove that $p$-solvable groups possess $p$-bases of size $3$ for every prime $p$. For other prominent groups we exhibit $p$-bases of size $2$.
In fact, we conjecture the existence of $p$-bases of size $2$ for every finite group. Finally, the notion of $p$-bases is generalized to blocks and fusion systems. 
\end{abstract}

\textbf{Keywords:} base, $p$-nilpotent centralizer, fusion\\
\textbf{AMS classification:} 20D20, 20B05

\section{Introduction}
Many algorithms in computational group theory depend on the existence of small bases. Here, a \emph{base} of a permutation group $G$ acting on a set $\Omega$ is a subset $\Delta\subseteq\Omega$ such that the pointwise stabilizer $G_\Delta$ is trivial (i.\,e. if $g\in G$ fixes every $\delta\in\Delta$, then $g=1$). 
The aim of this short note is to introduce a generalized base without the presence of a group action. 
To this end let us first consider a finite group $G$ acting faithfully by automorphisms on a $p$-group $P$. If $p$ does not divide $|G|$, then $G$ always admits a base of size $2$ by a theorem of Halasi--Podoski~\cite{base2}. Now suppose that $G$ is $p$-solvable, $P$ is elementary abelian and $G$ acts completely reducibly on $P$. Then $G$ has a base of size $3$ ($2$ if $p\ge 5$) by Halasi--Maróti~\cite{HalasiMaroti}. In those situations we may form the semidirect product $H:=P\rtimes G$. Now there exists $\Delta\subseteq P$ such that $|\Delta|\le 3$ and $\C_H(\Delta)=\C_H(\langle\Delta\rangle)\le P$. This motivates the following definition.

\begin{Def}
Let $G$ be a finite group with Sylow $p$-subgroup $P$. A subset $\Delta\subseteq P$ is called a \emph{$p$-base} of $G$ if $\C_G(\Delta)$ is $p$-nilpotent, i.\,e. $\C_G(\Delta)$ has a normal $p$-complement.
\end{Def}

Clearly, any generating set of $P$ is a $p$-base of $G$ since $\C_G(P)=\Z(P)\times\pcore_{p'}(\C_G(P))$ (this observation is generalized in \autoref{lemsc} below). 

Our main theorem extends the work of Halasi--Maróti as follows.

\begin{Thm}\label{main}
Every $p$-solvable group has a $p$-base of size $3$ ($2$ if $p\ge 5$).
\end{Thm}

Although Halasi--Maróti's Theorem does not extend to non-$p$-solvable groups, the situation for $p$-bases seems more fortunate.
For instance, if $V$ is a finite vector space in characteristic $p$, then every base of $\GL(V)$ (under the natural action) contains a \emph{basis} of $V$, so its size is at least $\dim V$. On the other hand, $G=\AGL(V)=V\rtimes\GL(V)$ possesses a $p$-base of size $2$. 
To see this, let $P$ be the Sylow $p$-subgroup of $\GL(V)$ consisting of the upper unitriangular matrices. Let $x\in P$ be a Jordan block of size $\dim V$. Then $\C_{\GL(V)}(x)\le P\Z(\GL(V))$. For any $y\in\C_V(x)\setminus\{1\}$ we obtain a $p$-base $\Delta:=\{x,y\}$ such that $\C_G(\Delta)\le VP$. We have even found a $p$-base consisting of \emph{commuting} elements. 
After checking many more cases, we believe that the following might hold.

\begin{Con}\label{con1}
Every finite group has a (commutative) $p$-base of size $2$ for every prime $p$. 
\end{Con}

The role of the number $2$ in \autoref{con1} appears somewhat arbitrary at first. There is, however, an elementary dual theorem:
A finite group is $p$-nilpotent if and only if every $2$-generated subgroup is $p$-nilpotent. This can be deduced from the structure of minimal non-$p$-nilpotent groups (see \cite[Satz~IV.5.4]{Huppert}). It is a much deeper theorem of Thompson~\cite{ThompsonAut} that the same result holds when “$p$-nilpotent” is replaced by “solvable”. Similarly, $2$-generated subgroups play a role in the Baer--Suzuki Theorem and its variations. 

Apart from \autoref{main} we give some more evidence of \autoref{con1}.

\begin{Thm}\label{thm2}
Let $G$ be a finite group with Sylow $p$-subgroup $P$. Then \autoref{con1} holds for $G$ in the following cases:
\begin{enumerate}[(i)]
\item $P$ is abelian.
\item $G$ is a symmetric group or an alternating group.
\item $G$ is a general linear group, a special linear group or a projective special linear group.
\item $G$ is a sporadic simple group or an automorphism group thereof.
\end{enumerate}
\end{Thm}

Our results on (almost) simple groups carry over to the corresponding quasisimple groups by \autoref{lemquot} below.
The notion of $p$-bases generalizes to blocks of finite groups and even to fusion systems. 

\begin{Def}\hfill
\begin{itemize}
\item Let $B$ be a $p$-block of a finite group $G$ with defect group $D$. A subset $\Delta\subseteq D$ is called \emph{base} of $B$ if $B$ has a nilpotent Brauer correspondent in $\C_G(\Delta)$ (see \cite[Definition~IV.5.38]{AKO}). 

\item Let $\mathcal{F}$ be a saturated fusion system on a finite $p$-group $P$. A subset $\Delta\subseteq P$ is called \emph{base} of $\mathcal{F}$ if there exists a morphism $\phi$ in $\mathcal{F}$ such that $\phi(\langle\Delta\rangle)$ is fully $\mathcal{F}$-centralized and the centralizer fusion system $\mathcal{C}:=\C_{\mathcal{F}}(\phi(\langle\Delta\rangle))$ is trivial, i.\,e. $\mathcal{C}=\mathcal{F}_{\C_P(\Delta)}(\C_P(\Delta))$ (see \cite[Definition~I.5.3, Theorem~I.5.5]{AKO}). 
\end{itemize}
\end{Def}

By Brauer's third main theorem, the bases of the principal $p$-block of $G$ are the $p$-bases of $G$ (see \cite[Theorem~IV.5.9]{AKO}). Moreover, if $\mathcal{F}$ is the fusion system attached to an arbitrary block $B$, then the bases of $B$ are the bases of $\mathcal{F}$ (see \cite[Theorem~IV.3.19]{AKO}). By the existence of exotic fusion systems, the following conjecture strengthens \autoref{con1}.

\begin{Con}\label{conbase}
Every saturated fusion system has a base of size $2$. 
\end{Con}

We show that \autoref{conbase} holds for $p$-groups of order at most $p^4$.

\section{Results}

\begin{proof}[Proof of \autoref{main}]
Let $G$ be a $p$-solvable group with Sylow $p$-subgroup $P$. Let $N:=\pcore_{p'}(G)$. For $Q\subseteq P$, $\C_G(Q)N/N$ is contained in $\C_{G/N}(QN/Q)$. Hence, $\C_G(Q)$ is $p$-nilpotent whenever $\C_{G/N}(QN/Q)$ is $p$-nilpotent. Thus, we may assume that $N=1$. Instead we consider $N:=\pcore_p(G)$. Since $G$ is $p$-solvable, $N\ne 1$. We show by induction on $|N|$ that there exists a $p$-base $\Delta\subseteq N$ such that $\C_G(\Delta)\le N$. By the Hall--Higman lemma (see \cite[Hilfssatz~VI.6.5]{Huppert}), $\C_{G/N}(N/\Phi(N))=N/\Phi(N)$ where $\Phi(N)$ denotes the Frattini subgroup of $N$. 
It follows that $\pcore_{p'}(G/\Phi(N))=1$.
Hence, by induction we may assume that $N$ is elementary abelian. Then $\overline{G}:=G/N$ acts faithfully on $N$ and it suffices to find a $p$-base $\Delta\subseteq N$ such that $\C_{\overline{G}}(\Delta)=1$. Thus, we may assume that $G=N\rtimes H$ where $\C_G(N)=N$ and $\pcore_p(H)=1$. 

Note that $\Phi(G)\le \F(G)=N$ where $\F(G)$ is the Fitting subgroup of $G$. Since $H$ is contained in a maximal subgroup of $G$, we even have $\Phi(G)<N$. Let $K\unlhd H$ be the kernel of the action of $H$ on $N/\Phi(G)$. By way of contradiction, suppose that $K\ne 1$. 
Since $K$ is $p$-solvable and $\pcore_p(K)\le\pcore_p(H)=1$, also $K_0:=\pcore_{p'}(K)\ne 1$. Now $K_0$ acts coprimely on $N$ and we obtain 
\[N=[K_0,N]\C_N(K_0)=\Phi(G)\C_N(K_0)\]
as is well-known. Both $\Phi(G)$ and $\C_N(K_0)H$ lie in a maximal subgroup $M$ of $G$. But then $G=NH=\Phi(G)\C_N(K_0)H\le M$, a contradiction. Therefore, $H$ acts faithfully on $N/\Phi(G)$ and we may assume that $\Phi(G)=1$. 
Then there exist maximal subgroups $M_1,\ldots,M_n$ of $G$ such that $N_i:=M_i\cap N<N$ for $i=1,\ldots,n$ and $\bigcap_{i=1}^nN_i=1$. Since $G=M_iN$, the quotients $N/N_i$ are simple $\FF_pH$-modules and $N$ embeds into $N/N_1\times\ldots\times N/N_n$. Hence, the action of $H$ on $N$ is faithful and completely reducible. Now by the main result of Halasi--Maróti~\cite{HalasiMaroti} there exists a $p$-base with the desired properties.
\end{proof}

Next we work towards \autoref{thm2}. 

\begin{Lem}\label{lemsc}
Let $P$ be a Sylow $p$-subgroup of $G$. Let $Q\unlhd P$ such that $\C_P(Q)\le Q$. Then every generating set of $Q$ is a $p$-base of $G$.
\end{Lem}
\begin{proof}
Since $P\in\Syl_p(\N_G(Q))$, we have $\Z(Q)=\C_P(Q)\in\Syl_p(\C_G(Q))$ and therefore $\C_G(Q)=\Z(Q)\times\pcore_{p'}(\C_G(Q))$ by the Schur--Zassenhaus Theorem. 
\end{proof}

\begin{Lem}\label{lemquot}
Let $\Delta$ be a $p$-base of $G$ and let $N\le\Z(G)$. Then $\overline{\Delta}:=\{xN:x\in\Delta\}$ is a $p$-base of $G/N$.
\end{Lem}
\begin{proof}
Let $gN\in\C_{G/N}(\overline{\Delta})$. Then $g$ normalizes the nilpotent group $\langle\Delta\rangle N$. Hence, $g$ acts on the unique Sylow $p$-subgroup $P$ of $\langle\Delta\rangle N$. Since $g$ centralizes \[\langle\overline{\Delta}\rangle=\langle\Delta\rangle N/N=PN/N\cong P/P\cap N\]
and $P\cap N\le N\le\Z(G)$, $g$ induces a $p$-element in $\Aut(P)$ and also in $\Aut(\langle\Delta\rangle N)$. 
Consequently, there exists a $p$-subgroup $Q\le\N_G(\langle\Delta\rangle N)$ such that $\C_{G/N}(\overline{\Delta})=Q\C_G(\Delta N)/N=Q\C_G(\Delta)/N$. Since $\C_G(\Delta)$ is $p$-nilpotent, so is $Q\C_G(\Delta)$ and the claim follows. 
\end{proof}

The following implies the first part of \autoref{thm2}.

\begin{Prop}\label{abel}
Let $P$ be a Sylow $p$-subgroup of $G$ with nilpotency class $c$. Then $G$ has a $p$-base of size $2c$.
\end{Prop}
\begin{proof}
The $p'$-group $\N_G(\Z(P))/\C_G(\Z(P))$ acts faithfully on $\Z(P)$. By Halasi--Podoski~\cite{base2} there exists $\Delta_0=\{x,y\}\subseteq\Z(P)$ such that $\N_H(\Z(P))\le\C_H(\Z(P))$ where $H:=\C_G(\Delta_0)$. If $c=1$, then $P=\Z(P)$ is abelian and Burnside's transfer theorem implies that $H$ is $p$-nilpotent. Hence, let $c>1$. 
By a well-known fusion argument of Burnside, elements of $\Z(P)$ are conjugate in $H$ if and only if they are conjugate in $\N_H(\Z(P))$. Consequently, all elements of $\Z(P)$ are isolated in our situation.
By the $\Z^*$-Theorem (assuming the classification of finite simple groups), we obtain 
\[\Z(H/\pcore_{p'}(H))=\Z(P)\pcore_{p'}(H)/\pcore_{p'}(H).\] 
The group $\overline{H}:=H/\Z(P)\pcore_{p'}(H)$ has Sylow $p$-subgroup $\overline{P}\cong P/\Z(P)$ of nilpotency class $c-1$. By induction on $c$ there exists a $p$-base $\overline{\Delta_1}\subseteq\overline{P}$ of $\overline{H}$ of size $2(c-1)$. We may choose $\Delta_1\subseteq P$ such that $\overline{\Delta_1}=\{\overline{x}:x\in\Delta_1\}$. Since $\overline{\C_H(\Delta_1)}\le\C_{\overline{H}}(\overline{\Delta_1})$ is $p$-nilpotent, so is
\[\bigl(\C_H(\Delta_1)\Z(P)\pcore_{p'}(H)/\pcore_{p'}(H)\bigr)/\Z(H/\pcore_{p'}(H)).\]
It follows that $\C_H(\Delta_1)\Z(P)\pcore_{p'}(H)/\pcore_{p'}(H)$ and $\C_H(\Delta_1)=\C_G(\Delta_0\cup\Delta_1)$ are $p$-nilpotent as well.
Hence, $\Delta:=\Delta_0\cup\Delta_1$ is a $p$-base of $G$ of size (at most) $2c$.
\end{proof}

\begin{Prop}\label{symalt}
The symmetric and alternating groups $S_n$ and $A_n$ have commutative $p$-bases of size $2$ for every prime $p$.
\end{Prop}
\begin{proof}
Let $n=\sum_{i=0}^k a_ip^i$ be the $p$-adic expansion of $n$. Suppose first that $G=S_n$. Let \[x=\prod_{i=0}^k\prod_{j=1}^{a_i}x_{ij}\in G\] 
be a product of disjoint cycles $x_{ij}$ where $x_{ij}$ has length $p^i$ for $j=1,\ldots,a_i$. Then $x$ is a $p$-element and 
\[\C_G(x)\cong\prod_{i=0}^kC_{p^i}\wr S_{a_i}.\] 
Since $a_i<p$, $P:=\langle x_{ij}:i=0,\ldots,k,j=1,\ldots,a_i\rangle$ is an abelian Sylow $p$-subgroup of $\C_G(x)$. Let $y:=\prod_{i=0}^k\prod_{j=1}^{a_i}x_{ij}^{j}\in P$. It is easy to see that $\Delta:=\{x,y\}$ is a commutative $p$-base of $G$ with
$\C_G(\Delta)=P$.

Now let $G=A_n$. If $p>2$, then $x,y$ lie in $A_n$ as constructed above and the claim follows from $\C_{A_n}(\Delta)\le\C_{S_n}(\Delta)$. Hence, let $p=2$. If $\sum_{i=1}^ka_i\equiv 0\pmod{2}$, then we still have $x\in A_n$ and $\C_G(x)=\langle x_{ij}:i,j\rangle$ is already a $2$-group. Thus, we have a $2$-base of size $1$ in this case. 
In the remaining case, let $m\ge 1$ be minimal such that $a_m=1$. We adjust our definition of $x$ by replacing $x_{m1}$ with a disjoint product of two cycles of length $2^{m-1}$. Then $x\in A_n$ and $\C_G(x)$ is a $2$-group or a direct product of a $2$-group and $S_3$ (the latter case happens if and only if $m=1=a_0$). We clearly find a $2$-element $y\in\C_G(x)$ such that $\C_G(x,y)$ is a $2$-group.
\end{proof}

The following elementary facts are well-known, but we provide proofs for the convenience of the reader.

\begin{Lem}\label{cyclo}
Let $p$ be a prime and let $q$ be a prime power such that $p\nmid q$. Let $e\mid p-1$ be the multiplicative order of $q$ modulo $p$. Let $p^s$ be the $p$-part of $q^e-1$. Then for every $n\ge 1$ the polynomial $X^{p^n}-1$ decomposes as
\[X^{p^n}-1=(X-1)\prod_{k=1}^{(p^s-1)/e}\gamma_{0,k}\prod_{i=1}^{n-s}\prod_{k=1}^{\phi(p^s)/e}\gamma_{i,k}\]
where the $\gamma_{i,k}$ are pairwise coprime polynomials in $\FF_q[X]$ of degree $ep^i$ for $i=0,\ldots,n-s$. 
\end{Lem}
\begin{proof}
Let $\zeta$ be a primitive root of $X^{p^n}-1$ in some finite field extension of $\FF_q$. Then \[X^{p^n}-1=\prod_{k=0}^{p^n-1}(X-\zeta^k).\] 
Recall that $\FF_q$ is the fixed field under the Frobenius automorphism $c\mapsto c^q$. Hence, the irreducible divisors of $X^{p^n}-1$ in $\FF_q[X]$ correspond to the orbits of $\langle q+p^n\ZZ\rangle$ on $\ZZ/p^n\ZZ$ via multiplication. The trivial orbit corresponds to $X-1$. For $i=1,\ldots,s$ the order of $q$ modulo $p^i$ is $e$ by the definition of $s$. This yields $(p^s-1)/e$ non-trivial orbits of length $e$ in $p^{n-s}\ZZ/p^n\ZZ$. The corresponding irreducible factors are denoted $\gamma_{0,k}$ for $k=1,\ldots,(p^s-1)/e$. 

For $i\ge 1$ the order of $q$ modulo $p^{s+i}$ divides $ep^i$ (it can be smaller if $p=2$ and $s=1$).
We partition $(p^{n-s-i}\ZZ/p^n\ZZ)^\times$ into $\phi(p^{s+i})/(ep^i)=\phi(p^s)/e$ unions of orbits under $\langle q+p^n\ZZ\rangle$ such that each union has size $ep^i$. The corresponding polynomials $\gamma_{i,1},\ldots,\gamma_{i,\phi(p^s)/e}$ are pairwise coprime (but not necessarily irreducible).
\end{proof}

\begin{Lem}\label{cent}
Let $A$ be an $n\times n$-matrix over an arbitrary field $F$ such that the minimal polynomial of $A$ has degree $n$. Then every matrix commuting with $A$ is a polynomial in $A$.  
\end{Lem}
\begin{proof}
By hypothesis, $A$ is similar to a companion matrix. Hence, there exists a vector $v\in F^n$ such that $\{v,Av,\ldots,A^{n-1}v\}$ is a basis of $F^n$. Let $B\in F^{n\times n}$ such that $AB=BA$. There exist $a_0,\ldots,a_{n-1}\in F$ such that $Bv=a_0v+\ldots+a_{n-1}A^{n-1}v$. Set $\gamma:=a_0+a_1X+\ldots+a_{n-1}X^{n-1}$. Then
\[BA^iv=A^iBv=a_0A^iv+\ldots+a_{n-1}A^{n-1}A^iv=\gamma(A)A^iv\]
for $i=0,\ldots,n-1$. Since $\{v,Av,\ldots,A^{n-1}v\}$ is a basis, we obtain $B=\gamma(A)$ as desired.
\end{proof}

\begin{Prop}\label{gl}
The groups $\GL(n,q)$, $\SL(n,q)$ and $\PSL(n,q)$ possess commutative $p$-bases of size $2$ for every prime $p$.
\end{Prop}
\begin{proof}
Let $q$ be a prime power. By \autoref{lemquot}, it suffices to consider $\GL(n,q)$ and $\SL(n,q)$. 
Suppose first that $p\mid q$. Let $x\in G:=\GL(n,q)$ be a Jordan block of size $n\times n$ with eigenvalue $1$. Then $x$ is a $p$-element since $x^{p^n}-1=(x-1)^{p^n}=0$. Moreover, $\C_G(x)$ consists of polynomials in $x$ by \autoref{cent}. In particular, $\C_G(x)$ is abelian and therefore $p$-nilpotent. Hence, we found a $p$-base of size $1$. Since $(q-1,p)=1$, this is also a $p$-base of $\SL(n,q)$. 

Now let $p\nmid q$. We “linearize” the argument from \autoref{symalt}. Let $e$ and $s$ be as in \autoref{cyclo}. Let $0\le a_0\le e-1$ such that $n\equiv a_0\pmod{e}$. Let 
\[\frac{n-a_0}{e}=\sum_{i=0}^{r}a_{i+1}p^i\] 
be the $p$-adic expansion. Let $M_i\in\GL(ep^i,q)$ be the companion matrix of the polynomial $\gamma_{i,1}$ from \autoref{cyclo} for $i=0,\ldots,r$. Let $G_i:=\GL(ea_{i+1}p^i,q)$ and $x_i:=\diag(M_i,\ldots,M_i)\in G_i$. Then the minimal polynomial of 
\[x:=\diag(1_{a_0},x_0,\ldots,x_r)\in G\] 
divides $X^{p^{r+s}}-1$ by \autoref{cyclo}. In particular, $x$ is a $p$-element. Since the $\gamma_{i,1}$ are pairwise coprime, it follows that
\[\C_G(x)=\GL(a_0,q)\times\prod_{i=0}^r\C_{G_i}(x_i).\]
Since $a_0<e$, $\GL(a_0,q)$ is a $p'$-group. 
By \autoref{cent}, every matrix commuting with $M_i$ is a polynomial in $M_i$. Hence, the elements of $\C_{G_i}(x_i)$ have the form $A=(A_{kl})_{1\le k,l\le a_{i+1}}$ where each block $A_{kl}$ is a polynomial in $M_i$. 
We define 
\[y_i:=\diag(M_i,M_i^2,\ldots, M_i^{a_{i+1}})\in\C_{G_i}(x_i)\]
and $y:=\diag(1_{a_0},y_0,\ldots,y_r)\in\C_G(x)$. Let $A=(A_{kl})\in\C_{G_i}(x_i,y_i)$. We want to show that $A_{kl}=0$ for $k\ne l$. To this end, we may assume that $k<l$ and $A_{kl}=\rho(M_i)$ where $\rho\in\FF_q[X]$ with $\deg(\rho)<\deg(\gamma_{i,1})=ep^i$.
Since $A\in\C_{G_i}(x_i,y_i)$, we have $M_i^kA_{kl}=M_i^lA_{kl}$ and $(M^{l-k}-1)A_{kl}=0$. It follows that the minimal polynomial $\gamma_{i,1}$ of $M_i$ divides $(X^{l-k}-1)\rho$. By way of contradiction, we assume that $\rho\ne 0$. Then $\gamma_{i,1}$ divides $X^{l-k}-1$ and $X^{p^{r+s}}-1$. However, $l-k\le a_{i+1}<p$ and $\gamma_{i1}$ must divide $X-1$. This contradicts the definition of $\gamma_{i,1}$ in \autoref{cyclo}. Hence, $A_{kl}=0$ for $k\ne l$. We have shown that the elements of $\C_G(x,y)$ have the form 
\[L\oplus\bigoplus_{i=0}^r\bigoplus_{j=1}^{a_{i+1}}L_{ij}\] 
where $L\in\GL(a_0,q)$ and each $L_{ij}$ is a polynomial in $M_i$. In particular, $\C_G(x,y)$ is a direct product of a $p'$-group and an abelian group. Consequently, $\C_G(x,y)$ is $p$-nilpotent.

Now let $G:=\SL(n,q)$. If $p\nmid q-1$, then the $p$-base of $\GL(n,q)$ constructed above already lies in $G$. Thus, we may assume that $p\mid q-1$. Then $e=1$ and $a_0=0$ with the notation above. We now have the polynomials $\gamma_{i,k}$ with $i=0,\ldots,r$ and $k=1,\ldots,p-1\le\phi(p^s)$ at our disposal. Let $M_{i,k}$ be the companion matrix of $\gamma_{i,k}$. 
Define 
\[x_i:=\diag(M_{i,1},\ldots,M_{i,a_{i+1}})\] 
for $i=0,\ldots,r$. Then the minimal polynomial of $x:=\diag(x_0,\ldots,x_r)\in\GL(n,q)$ has degree $n$ and therefore $\C_{\GL(n,q)}(x)$ is abelian by \autoref{cent}. Let $i\ge 0$ be minimal such that $a_{i+1}>0$. We replace the block $M_{i,1}$ of $x$ by the companion matrix of $X^{p^i}-1$. Then by \autoref{cyclo}, the minimal polynomial of $x$ still has degree $n$. Moreover, $x$ has at least one block $B$ of size $1\times 1$. We may modify $B$ such that $\det(x)=1$. After doing so, it may happen that $B$ occurs twice in $x$. In this case, $\C_G(x)\le \GL(2,q)\times H$ where $H$ is abelian. Then the matrix
\[y:=\begin{cases}
\begin{pmatrix}0&-1\\1&0\end{pmatrix}\oplus 1_{n-2}&\text{if }p=2,\\%
\diag(M_{0,1},M_{0,1}^{-1},1_{n-2})&\text{if }p>2
\end{cases}
\]
lies in $\C_G(x)$ and $\C_G(x,y)$ is abelian. Hence, $\{x,y\}$ is a $p$-base of $G$.
\end{proof}

\autoref{gl} can probably be generalized to classical groups. The next result completes the proof of \autoref{thm2}.

\begin{Prop}
Let $S$ be a sporadic simple group and $G\in\{S,S.2\}$. Then $G$ has a commutative $p$-base of size $2$ for every prime $p$.
\end{Prop}
\begin{proof}
If $p^4$ does not divide $|G|$, then the claim follows from \autoref{lemsc}.
So we may assume that $p^4$ divides $|G|$. From the character tables in the Atlas~\cite{Atlas} we often find $p$-elements $x\in G$ such that $\C_G(x)$ is already a $p$-group. In this case we found a $p$-base of size $1$ and we are done.
If $G$ admits a permutation representation of “moderate” degree (including $Co_1$), then the claim can be shown directly in GAP~\cite{GAP48}. In the remaining cases we use the Atlas to find $p$-elements with small centralizers:
\begin{itemize}
\item $G=Ly$, $p=2$: There exists an involution $x\in G$ such that $\C_G(x)=2.A_{11}$. By the proof of \autoref{symalt}, there exists $y\in A_{11}$ such that $\C_{A_{11}}(y)$ is a $2$-group. We identify $y$ with a preimage in $\C_G(x)$. Then $\C_G(x,y)$ is a $2$-group.

\item $G=Ly$, $p=3$: Here we find $x\in G$ of order $3$ such that $\C_G(x)=3.McL$. Since $McL$ contains a $3$-element $y$ such that $\C_{McL}(y)$ is a $3$-group, the claim follows.

\item $G=Th$, $p=2$: There exists an involution $x\in G$ such that $\C_G(x)=2^{1+8}_+.A_9$. As before we find $y\in\C_G(x)$ such that $\C_G(x,y)$ is a $2$-group.

\item $G=M$, $p=5$: There exists a $5$-element $x\in G$ such that $\C_G(x)=C_5\times HN$. Since there is also a $5$-element $y\in HN$ such that $\C_{HN}(y)$ is a $5$-group, the claim follows.

\item $G=M$, $p=7$: In this case there exists a radical subgroup $Q\le G$ such that $\C_G(Q)=Q\cong C_7\times C_7$ by Wilson~\cite[Theorem~7]{WilsonLocalM} (this group was missing in the list of local subgroups in the Atlas). Any generating set of $Q$ of size $2$ is a desired $p$-base of $G$. 

\item $G=HN.2$, $p=3$: There exists an element $x\in G$ of order $9$ such that $|\C_G(x)|=54$. Clearly, we find $y\in\C_G(x)$ such that $\C_G(x,y)$ is $3$-nilpotent. \qedhere
\end{itemize}
\end{proof}

Finally, we consider a special case of \autoref{conbase}.

\begin{Prop}
Let $\mathcal{F}$ be a saturated fusion system on a $p$-group $P$ of order at most $p^4$. Then $\mathcal{F}$ has a base of size $2$.
\end{Prop}
\begin{proof}
Recall that $A:=\Out_{\mathcal{F}}(P)$ is a $p'$-group and there is a well-defined action of $A$ on $P$ by the Schur--Zassenhaus Theorem. If $\mathcal{F}$ is the fusion system of the group $P\rtimes A$, then the claim follows from Halasi--Podoski~\cite{base2} as before. We may therefore assume that $P$ contains an $\mathcal{F}$-essential subgroup. In particular, $P$ is non-abelian. Let $Q<P$ be a maximal subgroup of $P$ containing $\Z(P)$. The fusion system $\C_{\mathcal{F}}(Q)$ on $\C_P(Q)=\Z(Q)$ is trivial by definition. Hence, we are done whenever $Q$ is generated by two elements. 

It remains to deal with the case where $|P|=p^4$ and all maximal subgroups containing $\Z(P)$ are elementary abelian of rank $3$. Since two such maximal subgroups intersect in $\Z(P)$, we obtain that $|\Z(P)|=p^2$ and $|P'|=p$ by \cite[Lemma~1.9]{Oliverindexp}, for instance. By the first part of the proof, we may choose an $\mathcal{F}$-essential subgroup $Q$ such that $\Z(P)<Q<P$. Let $A:=\Aut_{\mathcal{F}}(Q)$. 
Since $Q$ is essential, $P/Q$ is a non-normal Sylow $p$-subgroup of $A$ (see \cite[Proposition~I.2.5]{AKO}). Moreover, $[P,Q]=P'$ has order $p$. By \cite[Lemma~1.11]{Oliverindexp}, there exists an $A$-invariant decomposition 
\[Q=\langle x,y\rangle\times\langle z\rangle.\] 
We may choose those elements such that $\Delta:=\{xz,y\}\nsubseteq\Z(P)$. Then $\C_P(\Delta)=Q$ and $\C_A(\Delta)=1$. Let $\phi:S\to T$ be a morphism in $\mathcal{C}:=\C_{\mathcal{F}}(\Delta)$ where $S,T\le Q$. Then $\phi$ extends to a morphism $\hat\phi:S\langle\Delta\rangle\to T\langle\Delta\rangle$ in $\mathcal{F}$ such that $\hat\phi(x)=x$ for all $x\in\langle\Delta\rangle$. Hence, if $S\le\langle\Delta\rangle$, then $\phi=\id$. Otherwise, $S\langle\Delta\rangle=Q$ and $\hat\phi\in\C_A(\Delta)=1$ since morphisms are always injective. In any case, $\mathcal{C}$ is the trivial fusion system and $\Delta$ is a base of $\mathcal{F}$.
\end{proof}

\section*{Acknowledgment}
The author is supported by the German Research Foundation (\mbox{SA 2864/1-2} and \mbox{SA 2864/3-1}).


\begin{thebibliography}{1}

\bibitem{AKO}
M. Aschbacher, R. Kessar and B. Oliver, \textit{Fusion systems in algebra and
  topology}, London Mathematical Society Lecture Note Series, Vol. 391,
  Cambridge University Press, Cambridge, 2011.

\bibitem{Atlas}
J.~H. Conway, R.~T. Curtis, S.~P. Norton, R.~A. Parker and R.~A. Wilson,
  \textit{ATLAS of finite groups}, Oxford University Press, Eynsham, 1985.

\bibitem{GAP48}
The GAP~Group, \textit{GAP -- Groups, Algorithms, and Programming, Version
  4.11.0}; 2020, (\url{http://www.gap-system.org}).

\bibitem{HalasiMaroti}
Z. Halasi and A. Mar\'{o}ti, \textit{The minimal base size for a {$p$}-solvable
  linear group}, Proc. Amer. Math. Soc. \textbf{144} (2016), 3231--3242.

\bibitem{base2}
Z. Halasi and K. Podoski, \textit{Every coprime linear group admits a base of
  size two}, Trans. Amer. Math. Soc. \textbf{368} (2016), 5857--5887.

\bibitem{Huppert}
B. Huppert, \textit{Endliche {G}ruppen. {I}}, Grundlehren der Mathematischen
  Wissenschaften, Vol. 134, Springer-Verlag, Berlin, 1967.

\bibitem{Oliverindexp}
B. Oliver, \textit{Simple fusion systems over {$p$}-groups with abelian
  subgroup of index {$p$}: {I}}, J. Algebra \textbf{398} (2014), 527--541.

\bibitem{ThompsonAut}
J.~G. Thompson, \textit{Nonsolvable finite groups all of whose local subgroups
  are solvable}, Bull. Amer. Math. Soc. \textbf{74} (1968), 383--437.

\bibitem{WilsonLocalM}
R.~A. Wilson, \textit{The odd-local subgroups of the {M}onster}, J. Austral.
  Math. Soc. Ser. A \textbf{44} (1988), 1--16.

\end{thebibliography}
\end{document}